\documentclass[12pt]{amsart}
\usepackage[hmargin=2.5cm,vmargin=2.5cm]{geometry}
\usepackage{amsmath}
\usepackage{amsfonts}
\usepackage{amssymb}
\usepackage{amsthm}
\usepackage[noadjust]{cite}
\usepackage{enumitem}
\usepackage{setspace}
\usepackage{amsthm}
\usepackage{graphicx}
\usepackage{float}
\usepackage{subcaption}
\usepackage{xcolor}
\usepackage{tikz}
\usepackage{cleveref}

\AtBeginDocument{
	\def\MR#1{}
}

\pagecolor{white}

\allowdisplaybreaks

\newcommand{\cF}{\mathcal{F}}

\newcommand{\Z}{\mathbb{Z}}

\newcommand{\bQ}{\mathbb{Q}}
\newcommand{\Q}{\mathbb{Q}}

\newtheorem{thm}{Theorem}
\numberwithin{thm}{section}

\newtheorem{conj}[thm]{Conjecture}
\newtheorem{prop}[thm]{Proposition}
\newtheorem{lemma}[thm]{Lemma}
\newtheorem{cor}[thm]{Corollary}

\newtheorem*{namedtheorem}{\theoremname}
\newcommand{\theoremname}{testing}

\theoremstyle{definition}

\newtheorem*{nameddef}{\defname}
\newcommand{\defname}{testing}

\theoremstyle{remark}
\newtheorem{rmk}[thm]{Remark}

\begin{document}
	\title{Special Alternating knots are Band Prime}
	\author{Joe Boninger and Joshua Evan Greene}
	\address{Department of Mathematics, Boston College, Chestnut Hill, MA}
	\email{boninger@bc.edu, joshua.greene@bc.edu}
	\thanks{JB was supported by the National Science Foundation under Award No.~2202704. \\
	JEG was supported by the National Science Foundation under Award No.~DMS-2304856.}
	\maketitle
	
	\begin{abstract}
		We prove that a special alternating knot does not decompose as a non-trivial band sum. This restricts concordances from special alternating knots, and we conjecture that special alternating knots are ribbon concordance minimal. We verify our conjecture in many cases. This work is motivated by another conjecture of Owens and the second author, which posits that the set of alternating knots is downward closed under ribbon concordance.
	\end{abstract}

	\section{Introduction}
	
	A diagram $D \subset S^2$ of a link $K \subset S^3$ is {\em alternating} if a traveller along $D$ alternates between passing over and under crossing points. It is {\em special} if Seifert's algorithm returns one of the chessboard surfaces of $D$ or, equivalently, if one of the chessboard surfaces of $D$ is orientable.
	
	A link $K$ is {\em alternating} if it admits an alternating diagram, and it is {\em special alternating} if it admits a diagram which is both special and alternating.
	Special alternating links and their canonical Seifert surfaces are building blocks for the set of all alternating links under the operation of Murasugi sum.
	They include $(2,n)$-torus links, and they can be obtained more generally from any bipartite planar graph by inverting the Tait graph construction.
	
	Suppose knots $K_0, K_1 \subset S^3$ are separated from each other by a 2-sphere $S^2$, and let $b : I \times I \to S^3$ be an embedded band such that $b^{-1}(K_i) = I \times \{i\}$, $i \in \{0,1\}$.  The {\em band sum} $K = K_0 \#_b K_1 \subset S^3$ is the knot $(K_0 \cup K_1 -b(I \times \partial I)) \cup b(\partial I \times I)$. This band sum is {\em trivial} if $b^{-1}(S^2) = I \times \{1/2\}$, so that $K_0 \#_b K_1$ is equivalent to the connected sum $K_0 \# K_1$. A knot $K$ is {\em band prime} if it cannot be written as a non-trivial band sum.
	
	The main result of this note is the following:
	
	\begin{thm}
		\label{thm:main}
		Special alternating knots are band prime.
	\end{thm}

	Our motivation for studying band sums is their relationship to ribbon concordance, which was introduced by Gordon in an influential paper \cite{gor81}. Two knots $K_0$, $K_1$ are (smoothly) {\em concordant} if there exists a smooth, properly embedded annulus $C \subset S^3 \times I$ such that $\partial C \cap (S^3 \times \{i\}) = K_i \times \{ i \}$. The annulus $C$ is called a {\em ribbon concordance from $K_1$ to $K_0$}, and we write $K_0 \leq K_1$, if the restriction $h|_C$ of the height function $h : S^3 \times I \to I$ is Morse and has no critical points of index two. Ribbon concordance generalizes the notion of a ribbon knot, since a knot is ribbon if and only if it is ribbon concordant to the unknot.
	
	Theorem \ref{thm:main} implies:
	
	\begin{cor}
		\label{cor:concord}
		Suppose that $K_1$ is a special alternating knot, and $K_0 \leq K_1$ by a ribbon concordance $C \subset S^3 \times I$ such that $h|_C$ has two critical points. Then $K_0 \simeq K_1$, and $C$ is isotopic to the trivial concordance $K_1 \times I$.
	\end{cor}
	
	\begin{proof}
		Morse theory implies $h|_C$ has one critical point of index $0$---a birth---and one of index $1$---a saddle. Equivalently, the knot $K_1$ is a band sum of $K_0$ and an unknot, and the result follows from Theorem \ref{thm:main}.
	\end{proof}
	
	Agol \cite{ago22} proved ribbon concordance forms a partial ordering on the set of all knots in $S^3$, resolving a long open conjecture of Gordon \cite[Conjecture 1.1]{gor81}. This ordering has important properties: for instance, if $K_0 \leq K_1$, then the Seifert genus of $K_0$ is less than or equal to that of $K_1$ \cite{zem19}; the Alexander polynomial of $K_0$ divides that of $K_1$ \cite{gil84,fp20}; and if $K_1$ is fibered, then so is $K_0$  \cite{miy18, zem19}. In this vein, the second author and Owens conjectured that if $K_0 \leq K_1$ and $K_1$ is alternating, then so is $K_0$ \cite[Conjecture 9]{grow23}.
	
	Motivated by this conjecture and by \Cref{cor:concord}, we propose:
	
	\begin{conj}
		\label{conj:spec_min}
		Special alternating knots are ribbon concordance minimal.
	\end{conj}
	\noindent
	That is, if $K_1$ is special alternating, and $K_0 \le K_1$, then $K_0 \simeq K_1$.
	Note that a positive answer to \cite[Conjecture 9]{grow23} does not directly imply one to \Cref{conj:spec_min}.
	On the other hand, \Cref{conj:spec_min} implies \Cref{thm:main}, because Miyazaki proved that ribbon concordance minimal knots are band prime: for a prime knot, this is \cite[Theorem 3]{miy98}, and for a composite knot, this additionally requires \cite[Theorem 1]{miy20}.	
		
	We report on some more evidence in support of \Cref{conj:spec_min}.  One piece of evidence follows quickly from work of Zemke \cite{zem19}:
	
	\begin{prop}
		\label{thm:ev_one}
		Suppose $K$ is a special alternating knot, and $K' \leq K$. Then
		$$
		\widehat{HFK}(K') \cong \widehat{HFK}(K).
		$$
		In particular, the two knots have the same genus, Alexander polynomial, and determinant.
	\end{prop}
	\noindent
	Here $\widehat{HFK}(K)$ denotes the hat flavor of the knot Floer homology of $K$ with $\mathbb{F}_2$ coefficients.
	
	More evidence draws on the notions of $\bQ$-anisotropy and transfinite nilpotence, which we recall in Section 3.
	Gordon proved that any knot which is $\Q$-anisotropic and transfinitely nilpotent is ribbon concordance minimal, and he asked whether alternating knots are transfinitely nilpotent \cite[Theorem 1.3 and Question 5.2]{gor81}.
	An elementary argument shows that a special alternating knot is $\bQ$-anisotropic (\Cref{lem:qanis} below).
	Thus, a positive answer to Gordon's question would imply a positive answer to \Cref{conj:spec_min}.
	Transfinite nilpotence is known in certain circumstances, which leads to the following result:
	
	\begin{prop}
		\label{thm:ev_two}
		Suppose $K$ is a special alternating knot such that
		\begin{enumerate}
			\item $K$ is fibered,
			\item the leading coefficient of its Alexander polynomial  is a prime power, or
			\item $K$ is a two-bridge knot.
		\end{enumerate}
		Then $K$ is ribbon concordance minimal.
	\end{prop}
	
	We close the introduction with a conjecture strengthening \Cref{thm:main} and \Cref{conj:spec_min} and a piece of evidence for it.
	A knot $K \subset S^3$ is {\em positive} if it admits an oriented diagram in which every crossing is positive.
	If we orient a special alternating diagram, then every crossing has the same sign, positive or negative.
	Thus, up to mirroring, every special alternating knot is a positive knot.
	Nakamura proved that special alternating knots are the knots which are both alternating and positive up to mirroring \cite[Theorem 1]{nak00}.
	Extending \Cref{conj:spec_min}, we posit:	
	\begin{conj}
		Positive knots are ribbon concordance minimal.
	\end{conj}
	
	\noindent
	Other examples of positive knots include torus knots, which Gordon showed are ribbon concordance minimal \cite[Remark before Theorem 1.3]{gor81}. More broadly:
	\begin{prop}[cf.~\protect\cite{bamo17}]
		\label{prop:ev_three}
	Fibered positive knots are ribbon concordance minimal.
	\end{prop}
	\begin{proof}
	Suppose that $K_0 \leq K_1$ and $K_1$ is fibered and positive.
	As noted above, $K_0$ is fibered and $g(K_0) \le g(K_1)$, where $g$ denotes the Seifert genus.
	Since $K_1$ is positive, work of Rudolph implies that $g(K_1) = g_4(K_1)$, where $g_4$ denotes the smooth $4$-genus \cite{rudolph93,rudolph99}.
	Hence
	$$
	g(K_1) = g_4(K_1) = g_4(K_0) \leq g(K_0) \leq g(K_1),
	$$
	implying that $g(K_0) = g(K_1)$.
	Miyazaki observed that a genus-preserving ribbon concordance between fibered knots is trivial \cite[Remark before Proposition 5]{miy18}.
	Hence $K_0 \simeq K_1$.
	\end{proof}
	
	\Cref{prop:ev_three} implies item (1) of \Cref{thm:ev_two}; below, we give an alternate proof of the latter using classical methods. Concordance of positive knots has also been studied by Baker \cite{bak16} and Stoimenow \cite{sto15}.	In particular, a positive solution to both \cite[Question 6.1]{gor81}, which strengthens the slice-ribbon conjecture, and to \Cref{conj:spec_min} would answer both \cite[Question 7]{bak16} and \cite[Question 4.2]{sto15}.
	\subsection*{Outline}
	
	In Section \ref{sec:two}, we prove Theorem \ref{thm:main}. In Section \ref{sec:three}, we prove Propositions \ref{thm:ev_one} and \ref{thm:ev_two}.
	
	\section{Proof of the Main Theorem}
	\label{sec:two}
	
	We begin by reducing the proof of \Cref{thm:main} to the case of prime special alternating knots; this is possible thanks to results of Eudave-Mu\~noz \cite{em92} and Menasco \cite{m84}.
	
	\begin{lemma}
		\label{lem:decompose}
		If a composite knot is not band prime, then one of its factors is not band prime.
	\end{lemma}
	
	\begin{proof}
		Suppose a composite knot $K$ can be written as a non-trivial band sum $K_0 \#_b K_1$. Eudave-Mu\~noz showed that there exists a decomposing sphere $F \subset S^3$ for $K$ (as a composite knot) which is disjoint from the band $b$ \cite[Theorem 2]{em92}.
		This implies
		$$
		2 = |F \cap K| = |F \cap (K_0 \sqcup K_1)|,
		$$
		and therefore $F$ is a (possibly trivial) decomposing sphere for either $K_0$ or $K_1$. We assume the latter without loss of generality, and let $K_1 = K_A \# K_B$ be the decomposition of $K_1$ determined by $F$. Since $F$ is disjoint from $b$, $b$ is contained on one side of $F$, which we assume is the side containing $K_A$. Equivalently, we have
		$$
		K = (K_0 \#_b K_A) \# K_B.
		$$
		If the band sum $K_0 \#_b K_A$ is trivial, then
		$$
		K = (K_0 \# K_A) \# K_B = K_0 \# (K_A \# K_B) = K_0 \# K_1.
		$$
		By an observation of Miyazaki following Eudave-Mu\~noz's work, this implies the band sum $K_0 \#_b K_1$ is trivial \cite[Theorem 1]{miy20}. This contradiction shows the band sum $K_0 \#_b K_A$ is nontrivial, so $K$ admits a factor which is not band prime.
	\end{proof}
	
	\begin{lemma}
		\label{lem:decompose2}
		If a special alternating knot $K$ decomposes as a connected sum $K_0 \# K_1$, then each of $K_0$ and $K_1$ are special alternating.
	\end{lemma}
	
	\begin{proof}
		Let $D$ be an alternating diagram of $K$. A theorem of Menasco implies that $K_i$ has an alternating diagram $D_i$ for $i=0,1$ such that $D = D_0 \# D_1$ \cite[Theorem 1]{m84}. It is straightforward to check that each of the $D_i$ is special if $D$ is.
	\end{proof}
	
	\begin{cor}
		\label{cor:prime}
		If there exists a special alternating knot which is not band prime, then there exists a prime special alternating knot which is not band prime. \qed
	\end{cor}
	
	
	\begin{proof}[Proof of Theorem \protect\ref{thm:main}]
		Let $K \subset S^3$ be a special alternating knot, and suppose toward a contradiction that $K$ decomposes as a non-trivial band sum $K_0 \#_b K_1$. By \Cref{cor:prime}, we may assume $K$ is prime.
		
		Gabai showed that there exists a minimal genus Seifert surface $S$ for $K$ and a properly embedded arc $\alpha \subset S$, coinciding with a co-core of $b$, such that $S' = S - \eta(\alpha)$ is a Seifert surface for the split link $K_0 \sqcup K_1$ \cite[Proof of Theorem 1]{gab87}. Here $\eta$ denotes a regular open neighborhood. 
		
		Hirasawa-Sakuma proved any minimal genus surface of a special alternating knot is a chessboard surface of an alternating diagram of that knot \cite[Theorem 1.1]{hs97}.
		Accordingly, we choose an alternating diagram $D \subset S^2$ of $K$ admitting $S$ as a chessboard surface.
		Let $G \subset S^2$ be the corresponding Tait graph, and let $\lambda_S$ denote the symmetrized Seifert pairing on $H_1(S)$.
		Since $S$ is a chessboard surface and $D$ is alternating, Gordon and Litherland proved in this case that $(H_1(S),\lambda_S)$ is isometric, up to an overall sign, to the lattice $\cF(G)$ of integer-valued flows on $G$ \cite[Theorem 1]{gl78}.
		In particular, $\lambda_S$ is definite.
		
		A Mayer-Vietoris argument shows the inclusion-induced map $H_1(S') \to H_1(S)$ injects, so it follows that $\lambda_{S'}$ is definite as well.
		Hence both $S$ and $S'$ are incompressible, by \cite[Corollary 3.2]{g17}.
		
		Let $F \subset S^3$ be a 2-sphere separating $K_0$ and $K_1$, chosen so $F$ and $S'$ are in general position and $F \cap S'$ has the minimal number of components. A standard innermost curve argument, using the incompressibility of $S'$, shows $F \cap S' = \varnothing$.
		Thus $S'$ has two components which are separated by the sphere $F$; we denote these by $S_0$ and $S_1$, so that $\partial S_i = K_i$. Viewing $S$ as $S_0 \cup_\alpha S_1$, we conclude that the inclusion map $S_0 \sqcup S_1 \hookrightarrow S$ induces a lattice isomorphism $(H_1(S_0),\lambda_{S_0}) \oplus (H_1(S_1),\lambda_{S_1}) \cong (H_1(S),\lambda_S)$.
		
		We claim each of $H_1(S_0)$ and $H_1(S_1)$ are non-zero. Suppose not; then $H_1(S_1)$, say, has rank $0$, so $S_1$ is a disk. In this case the boundary of a regular neighborhood of $S_1$ in $S^3 - K_0$ is a sphere which separates $K_0$ and $K_1$ and meets $S$ in a single arc parallel to $\alpha$, implying $b$ is trivial. This contradiction proves the claim, and it follows that $(H_1(S),\lambda_S)$ admits a non-trivial decomposition as an orthogonal direct sum.

		We have thereby shown that $\cF(G) \cong (H_1(S), \pm\lambda_S)$ admits a non-trivial decomposition as an orthogonal direct sum.
		By a result of Bacher, de la Harpe, and Nagnibeda, it follows that $G$ is {\em separable}: there exist subgraphs $G_0, G_1 \subset G$ with $\mathcal{F}(G_i) \cong (H_1(S_i),\lambda_{S_i})$, such that $G_0 \cup G_1 = G$ and $G_0 \cap G_1$ consists of a single vertex \cite[Proposition 4]{bhpn97}.
		Equivalently, $D$ admits a connected sum decomposition $D= D_0 \# D_1$ with $G_i$ a Tait graph of $D_i$.	
		Let $K_i \subset S^3$ be the knot represented by $D_i$, $i=0,1$.
		Each form $\lambda_{S_i}$ is non-zero and definite, so the signature satisfies $\sigma(K_i) = \sigma(\lambda_{S_i}) \ne 0$, implying each $K_i$ is non-trivial.
		We conclude that $K$ is not prime, a contradiction which proves the theorem.
	\end{proof}
	
	
	\begin{rmk}
		The proof of \Cref{thm:main} shows, in fact, that all non-split special alternating links are band prime. While the cited results of Eudave-Mu\~noz \cite{em92}, Gabai \cite{gab87}, and Miyazaki \cite{miy20} are stated only for knots, the remark in \cite{em92} following the proof of Theorem 2 can be used in place of \cite[Theorem 2]{em92}, and \cite[Proposition 8.9]{sch89} can be used in place of Gabai's work. Similarly, the proof of Miyazaki's result extends to non-split links.
		On the other hand, it would require more effort to extend the result to split special alternating links: one would like to know that the split union of band prime links is again band prime.
		We decided not to pursue this issue, besides to sense that it is non-trivial.
	\end{rmk}
	
	\begin{rmk}
		We had originally attempted to prove \cite[Conjecture 9]{grow23} using the characterization of alternating knots in terms of definite spanning surfaces \cite{g17}.
		However, this approach proved more challenging to implement than expected.
		The use of a definite chessboard surface in the proof of \Cref{thm:main} is the surviving remnant of this idea.
	\end{rmk}
	
	\section{Ribbon Concordance}
	\label{sec:three}
	
	In light of Corollary \ref{cor:concord}, Theorem \ref{thm:main} provides evidence that special alternating knots are ribbon concordance minimal. In this section we prove Propositions \ref{thm:ev_one} and \ref{thm:ev_two}, which provide more support for this conjecture.	
	\begin{proof}[Proof of \protect\Cref{thm:ev_one}]
		The relation $K' \leq K$ implies $\sigma(K') = \sigma(K)$.
		Zemke showed that for any ribbon concordance from $K$ to $K'$, the induced map $\widehat{HFK}(K') \to \widehat{HFK}(K)$ is injective \cite[Theorem 1.1]{zem19}.
		Since knot Floer homology detects genus, it follows that $g(K') \le g(K)$ \cite[Theorem 1.3]{zem19}.
		As in the proof of \Cref{thm:main}, since $K$ is special alternating, it satisfies $|\sigma(K)| = 2g(K)$. Thus
		$$
		|\sigma(K)| = |\sigma(K')| \leq 2g(K') \leq 2g(K) = |\sigma(K)|,
		$$
		and $g(K) = g(K')$. Additionally, because $K$ is alternating, $K$---and therefore $K'$---are homologically thin. 
		Letting $\Delta$ denote the Alexander polynomial, we obtain
		$$
		\text{span}(\Delta_{K'}) = 2g(K') = 2g(K) = \text{span}(\Delta_K).
		$$
		Next, we use the fact that $K' \leq K$ implies $\Delta_{K'}$ divides $\Delta_K$ \cite{gil84, fp20}. Since $\text{span}(\Delta_{K'}) =  \text{span}(\Delta_K)$, we have $m\Delta_{K'} = \Delta_K$ for some $m \in \Z$. Since $\Delta_{K}(1) = \Delta_{K'}(1) = 1$, we have $m = 1$ and $\Delta_{K'} = \Delta_K$. Finally, because $K'$ and $K$ are homologically thin and supported in the same $\delta$-grading, their knot Floer homology groups are determined by their Alexander polynomials. Thus,
		\[
		\widehat{HFK}(K') \cong \widehat{HFK}(K). \qedhere
		\]
	\end{proof}
	
	Before proving \Cref{thm:ev_two}, we review some terminology.
	Given a knot $K$, we let $C$ denote the commutator subgroup of the knot group $\pi_1(S^3 - K)$. The lower central series $\{\gamma_i(C)\}_{i \geq 1}$ of $C$ is defined recursively by
	\begin{align*}
		\gamma_1(C) &= C, \\
		\gamma_n(C) &= [G, \gamma_{n - 1}(G)],
	\end{align*}
	and for a limit ordinal $\omega$ we define $\gamma_\omega(G) := \bigcap_{\alpha < \omega} \gamma_\alpha(G)$. We say $C$ is {\em transfinitely nilpotent} if its lower central series is eventually trivial, and we say $K$ is transfinitely nilpotent if $C$ is. Free groups give the simplest examples of groups which are transfinitely nilpotent but not nilpotent \cite[Chapter 5]{mks04}.
	
	If $X$ denotes the universal abelian cover of $K$, then Milnor defined a quadratic form $\mu$ on $H^1(X; \Q)$ with rank rk$(\mu) = \text{span}(\Delta_K)$ and signature $\sigma(\mu) = \sigma(K)$ \cite{mil68}.
	Following Gordon, we say $K$ is {\em $\Q$-anisotropic} if $H^1(X; \Q)$ contains no non-trivial subspace which is annihilated by $\mu$ and invariant under deck transformations \cite[Proposition 4.5]{gor81}.
	
	\begin{lemma}
		\label{lem:qanis}
		Special alternating knots are $\bQ$-anisotropic.
	\end{lemma}
	
	\begin{proof}
		If $K$ is a special alternating knot, then we have
		$$
		|\sigma(\mu)| = |\sigma(K)| = \text{span}(\Delta_K) = \text{rk}(\mu),
		$$
		so $\mu$ is a definite form. It follows that there is no non-trivial subspace of $H^1(X; \Q)$ annhilated by $\mu$, let alone one which is invariant under deck transformations.
	\end{proof}
	
	\begin{prop}
		\label{prop:tn}
		If $K$ is a special alternating knot and $K$ is transfinitely nilpotent, then $K$ is ribbon concordance minimal.
	\end{prop}
	
	\begin{proof}
		Gordon proved that any knot which is $\Q$-anisotropic and transfinitely nilpotent is ribbon concordance minimal \cite[Theorem 1.3]{gor81}.
		Hence the result follows from \Cref{lem:qanis}.
		
		Alternatively, Gordon proved that if $K' \leq K$, $K$ is transfinitely nilpotent, and span$(\Delta_K) = \text{span}(\Delta_{K'})$, then $K' \simeq K$ \cite[Lemma 3.4]{gor81}. 
		Hence the result follows from \Cref{thm:ev_one}.
	\end{proof}

	\begin{proof}[Proof of \protect\Cref{thm:ev_two}]
		If $K$ satisfies any of the listed hypotheses, then $K$ is transfinitely nilpotent:
		\begin{enumerate}
			\item
			follows because commutator subgroups of fibered knots are free,
			\item
			is a classical result of Mayland and Murasugi \cite[Theorem A]{mm76}, and
			\item
			was proven more recently by Johnson \cite[Corollary 1.3 and Remark 1.20]{joh21}.
		\end{enumerate}
		Hence the result follows from \Cref{prop:tn}.
	\end{proof}


	\bibliography{main_bib}{}
	\bibliographystyle{amsplain}
	
\end{document}